\documentclass[journal, twocolumn]{IEEEtran}

\usepackage{color}

\ifCLASSINFOpdf
\else
\fi

\usepackage{mathptmx} 
\usepackage{amsmath} 
\usepackage{amssymb}  
\usepackage{amsthm}
\usepackage{slashbox}
\usepackage{graphicx}
\usepackage{dsfont}
\usepackage{mathtools}
\usepackage{framed}
\usepackage{url}

\newtheorem{thm}{Theorem}[section]

\newtheorem{lem}[thm]{Lemma}
\newtheorem{prop}[thm]{Proposition}

\newtheorem{defn}[thm]{Definition}
\newtheorem{rem}[thm]{Remark}

\newcommand{\mc} {\mathcal}

\newcommand{\md} {\mathds}
\newcommand{\ra} {\rightarrow}

\begin{document}
%
\title{Convex searches for discrete-time Zames--Falb multipliers}
%
%
%

\author{Joaquin Carrasco,~\IEEEmembership{Member,~IEEE,}
    	William P. Heath,~\IEEEmembership{Member,~IEEE,}
        Nur Syazreen Ahmad,~\IEEEmembership{Member,~IEEE,}
        Shuai Wang,~\IEEEmembership{Member,~IEEE}
        and~Jingfan Zhang,~\IEEEmembership{Student Member,~IEEE}
\thanks{J. Carrasco, W. P. Heath, and J. Zhang are with the School of Electrical and Electronic Engineering, University of Manchester, Sackville St. Building, Manchester M143 9PL, UK.  e-mail:  joaquin.carrascogomez@manchester.ac.uk, william.heath@manchester.ac.uk, jingfan.zhang@postgrad.manchester.ac.uk.}
\thanks{N. S. Ahmad is with the School Of Electrical \& Electronic Engineering,
USM, Engineering Campus,
Seberang Perai Selatan,
Nibong Tebal, Penang,
14300,
Malaysia.
  e-mail: syazreen@usm.my}
\thanks{S. Wang is a Senior Researcher with Robotics X in Tencent, High Technology Park, Nanshan District, Shenzhen,  518057,  China. e-mail: shawnshwang@tencent.com}
}

\maketitle

\begin{abstract}
In this paper we develop and analyse convex searches for Zames--Falb multipliers. We present two different approaches:  Infinite Impulse Response (IIR) and Finite Impulse Response (FIR) multipliers. The set of FIR multipliers is complete in that any IIR multipliers can be phase-substituted by an arbitrarily large order FIR multiplier. We show that searches in discrete-time for FIR multipliers are effective even for large orders. As expected, the numerical results provide the best $\ell_{2}$-stability results in the literature for slope-restricted nonlinearities. Finally, we demonstrate that the discrete-time search can provide an effective method to find suitable continuous-time multipliers.

\end{abstract}


%
\IEEEpeerreviewmaketitle

\section{Introduction}

The stability of a feedback interconnection between a linear time-invariant system $G$ and any nonlinearity $\phi$ within the class of nonlinearities $\Phi$ is referred to as the Lur'e problem (see Section 1.3 in~\cite{Aizerman:64} for a history of this problem). As the stability is obtained for the whole class of nonlinearities, the adjective ``absolute'' or ``robust'' is added. In the classical solution of this problem frequency-domain conditions on the linear system are determined by the class of nonlinearites. The inclusion of a multiplier reduces the conservativeness of the approach. The stability problem is translated into the search for a multiplier $M$ which belongs to the class of multipliers associated with the class of nonlinearities $\Phi$, where $G$ and $M$ satisfy some frequency conditions.

The class of Zames--Falb multipliers is defined both for the continuous-time domain~\cite{Zames:68} and for the discrete-time domain~\cite{Willems:68} (see \cite{Carrasco:16} for  a tutorial on Zames--Falb multipliers for the continuous-time domain). Loosely speaking, a Zames--Falb multiplier preserves the positivity of a monotone and bounded nonlinearity. Hence if an LTI plant $G$ is in negative feedback with a monotone and bounded nonlinearity then stability is guaranteed if there is a multiplier $M$ such that 
\begin{equation}\label{eq:0}
\mbox{Re}\{MG\}>0,
\end{equation}
with $M$ and $G$ evaluated over all frequencies (i.e. at $j\omega$, $\omega\in \mathbb{R}$ for continuous-time systems and at $e^{j\omega}$, $\omega\in[0,2\pi]$ for discrete-time systems). Similarly (and by loop tranformation) if an LTI plant $G$ is in negative feedback with an $S[0,k]$ slope-restricted nonlinearity, then stability is guaranteed if there is a multiplier $M$ such that 
\begin{equation}\label{eq:1}
\mbox{Re}\{M(1+kG)\}>0,
\end{equation}
with $M$ and $G$ evaluated over all frequencies. 
In addition a wider class of multipliers is available if the nonlinearity is odd;
multipliers for quasi-odd multipliers can also be derived \cite{Heath:15cdc}.

\subsection{Oveview of searches of Zames--Falb multipliers in Continuous-time}
To date, most of the literature on search methods for Zames-Falb multipliers has been focused on continuous-time systems, where three types of method have been developed:
\paragraph{Finite Impulse Response (FIR)}
Searches over sums of Dirac delta functions are proposed and developed in~\cite{Safonov87},~\cite{Gapski:94} and~\cite{Chang:2012}. The main advantage of this method is the simplicity and versatility of using impulse responses for the multiplier. However the searches require a sweep over all  frequencies, which can lead to unreliable results in some cases~\cite{Carrasco:2012a}. Moreover, the choice of times for the Dirac delta functions is heuristic.
\paragraph{Basis functions}In~\cite{chen95a} and~\cite{Chen:96} it is proposed to parameterise the multiplier in terms of causal basis functions $e_i^+(t)=t^ie^{-t}u(t)$ where $u(t)$ is the unit (or Heaviside) step function, and anticausal basis functions $e_i^-(t)=t^ie^tu(-t)$, with $i=1,\ldots,N$ for some $N$. As an advantage over the FIR method, the positivity of $M(1+kG)$ can be tested through the KYP lemma.  
Moreover the search provides a complete search over the class of rational multipliers as $N$ approaches infinity~\cite{Veenman:13}. The method provided significant advantages, such as the combination with other nonlinearities~\cite{Veenman:16}. Nonetheless if $N$ is required to be large then the search becomes numerically ill-conditioned. With small $N$ there is conservatism for odd nonlinearities, since the impulse of the multiplier is allowed to change sign. 
In fact the results reported in \cite{chen95a} for SISO examples are not significantly better for odd nonlinarities than for non-odd.
\paragraph{Restricted structure rational multipliers} In~\cite{Turner:2009} an LMI method is proposed where the $\mc{L}_1$ norm of a low-order causal multiplier is bounded in a convex manner (see also~\cite{Carrasco:2012}). 
Several extensions have been proposed: adding a Popov multiplier~\cite{Turner:11}, developing an anticausal counterpart~\cite{Carrasco:2012a}, and increasing the order of the multiplier~\cite{Turner:13}. The method is  quasi-convex and effective but does not provide a complete search. It has  two further drawbacks: the bound of the $\mc{L}_1$-norm may be conservative and it can only be applied if the nonlinearity is odd. 
\vskip3mm

In~\cite{Carrasco:14,Carrasco:16}, it has been shown that their relative performances vary with different examples. It must be highlighted that results in the basis functions can be significantly improved by manually selecting the parameters of the basis~\cite{Fetzer:17a,Tugal:17}. Similarly, manual tuning of delta functions can be useful for time-delay systems~\cite{Zhang:18b}.

In addition, there are several other stability tests in the literature, where either the Zames--Falb multipliers are not explicitly invoked or extensions to the Zames--Falb multipliers are proposed. These can all be viewed as searches over subclasses of Zames--Falb multipliers \cite{Carrasco:13,Carrasco:14}. In particular, the Off-Axis Circle Criterion is a powerful technique that uses graphical tools to ensure the existence of a possibly high-order multiplier by using graphical methods~\cite{Cho:68}, hence avoiding the use of an optimization tool. It can be used to establish a large set of plants that satisfy the Kalman conjecture~\cite{Tesi:92,Zhang:18}.

\subsection{Zames--Falb multipliers in Discrete-time domain}
In~\cite{Willems:68,Willems:71}, the discrete-time counterparts of the Zames--Falb multipliers~\cite{Zames:68} are given. The conditions are the natural counterparts to the continuous-time case, where the $\mathcal{L}_1$-norm is replaced by the $\ell_1$-norm and the frequency-domain inequality must be satisfied on the unit circle. In the continuous-time case, the use of improper multipliers has generated ``extensions'' of the original that have been analysed in~\cite{Carrasco:13,Carrasco:14}. In the discrete-time case, the conditions for the Zames--Falb multipliers are necessary and sufficient to preserve the positivity of the nonlinearity~\cite{Willems:71}; it follows that  the class of Zames--Falb multipliers is the widest class of multipliers that can be used. The result has been extended to MIMO systems~\cite{Safonov:00}, repeated nonlinearities in~\cite{Kulkarni:2002} and MIMO repeated nonlinearities in~\cite{Mancera:2004}. These works are focused on the description of the available multipliers, but no explicit search method is discussed.

Modern digital control implementation requires a complete study in the discrete-time domain. In addition the possibility of using the Zames--Falb multipliers for studying the stability and robustness properties of input-constrained model predictive control (MPC)~\cite{Heath:07} provides an inherent motivation for discrete-time analysis, since MPC is naturally  formulated in discrete time. Recently, Zames--Falb multipliers in discrete-time have been attracting attention in their use to ensure convergence rates of optimization algorithms~\cite{Lessard:16,Freeman:18}.

More generally, the absolute stability problem of  discrete-time Lur'e systems with slope--restricted nonlinearities continues to attract attention. Recent studies include~\cite{Gonzaga:2012,Ahmad:13j,Ahmad:15, Park:2018} which all take a  Lyapunov function approach; as an advantage they generate easy-to-check Linear Matrix Inequality (LMI) conditions. However one might expect  that improved results could be obtained via a multiplier approach, since this provides a more general condition. In fact some of these approaches can be interpreted as a search over a small subclass of Zames--Falb multipliers; see~\cite{Ahmad:15} for further details. Although this paper deals with SISO systems, it must be highlighted that a tractable stability test using Zames--Falb multipliers for MIMO nonlinearities has been proposed in~\cite{Fetzer:17b}. 

The differences between continuous-time and discrete-time Lur'e systems are non-trivial. As an example, second-order counterexamples to the discrete-time Kalman conjecture have been found \cite{Carrasco:15,Heath:15}. For continuous-time systems the Kalman conjecture holds for first, second, and third order plants~\cite{Barabanov88}. This is reflected by phase restrictions that can be placed on discrete-time Zames--Falb multipliers that are different in kind to their continuous-time counterparts \cite{Wang:18}.

In this paper we propose several searches for SISO LTI discrete-time Zames--Falb multipliers.   
The search of multipliers can be carried out with two different approaches:
\paragraph{Infinite impulse response (IIR) multiplier} The search is the counterpart of the method proposed by Turner et al. \cite{Turner:2009,Carrasco:2012a}, presented in~\cite{Ahmad:13} and included for the sake of completeness. The multipliers are parametrised in terms of their state-space representation, and classical multiobjective techniques are used to produce an LMI search.
\paragraph{Finite impulse response (FIR) multiplier} This search can be considered as the counterpart of both Safonov's and Chen and Wen's methods (\cite{Safonov87,Chen:96}). Initial results were presented in~\cite{Wang:14}. Here, two alternative versions are provided: firstly we propose an \emph{ad hoc} factorization which leads to a hard-factorization of the multiplier; secondly we use standard lifting techniques, e.g.~\cite{Hosoe:13}, whose factorization need not be hard but can provide other advantages.   
\vskip3mm

Numerical results and some computational consideration are discussed in Section V. 
In Section VI we consider how the discrete-time FIR search may be used effectively to find continuous-time multipliers. We show by numerical examples that tailoring the method can match or beat searches proposed in the literature for rational transfer functions. 

We must highlight that discrete-time Zames--Falb multipleirs have been defined as LTV operators~\cite{Willems:68}. However, we reduce our attention to LTI Zames--Falb multiplier. In the spirit of~\cite{Carrasco:13}, it remains open whether the restriction to LTI Zames--Falb multiplier can be made without loss of generality when $G$ is an LTI system. Moreover we have conjectured that if there is no suitable Zames--Falb multiplier for a plant $G$ and gain $k$ smaller than its Nyquist gain (see Section II for a definition), then there exists a slope-restricted nonlinearity in $[0,k]$ such that the feedback interconnection between $G$ and the nonlinearity is unstable~\cite{Wang:18}. However, further work is required to prove or disprove these conjectures.


\section{Notation and Preliminary results}


Let $\md{Z}$ and $\md{Z_+}$ be the set of integer numbers and positive integer numbers including $0$, respectively. Let $\ell$ be the space of all real-valued sequences, $h:\md{Z}_+\ra\md{R}$.  Let $\ell_1(\md{Z})$ be the space of all absolute summable sequences, so given a sequence $h:\md{Z}\ra\md{R}$ such that $h\in\ell_1$, then its $\ell_1$-norm is
\begin{equation}\label{eq:2}
\|h\|_1=\sum_{k=-\infty}^{\infty}|h_k|,
\end{equation}
where  $h_k$ means the $k$th element of $h$.
In addition, let $\ell_{2}$ denote the Hilbert space of all square-summable real sequences $f:\md{Z}_+\rightarrow\md{R}$ with the inner product defined as
\begin{equation}\label{eq:3}
\langle f, g\rangle=\sum_{k=0}^\infty f_k g_k,
\end{equation}
for $f,\ g\in\ell_2$, $k\in\md{Z}_+$. Similarly, we can define the Hilbert space $\ell_2(\md{Z})$ by considering real sequences $f:\md{Z}\rightarrow\md{R}$. We use $0_i$ to denote a row vector with $i$ entries, all equal to zero. Similarly $0$ denotes a matrix with zero entries where the dimension is obvious from the context. We use $I_i$ to denote the $i\times i$ identity matrix.

The standard notation $\mathbf{RL}_\infty$ is used for the space of all real rational transfer functions  with no poles on the unit circle. If $G\in\mathbf{RL}_\infty$, its norm is defined as $\|G\|_\infty=\sup_{|z|=1}|G(z)|$. Furthermore $\mathbf{RH}_\infty$ is used for the space of all real rational transfer functions with all poles strictly inside the unit circle. Similarly, $\mathbf{RH}_\infty^{-}$ is used for the space of all real rational transfer functions with all poles strictly outside the unit circle.  With some reasonable abuse of the notation, given a rational transfer function $H(z)$ analytic on the unit circle, $\|H\|_1$ means the $\ell_1$-norm of impulse response of $H(z)$.

Let $\bar{M}$ denote a linear time invariant operator mapping a time domain input signal to a time domain output signal and let $M$ denote the corresponding transfer function. We consider that the domain of convergence includes the unit circle, so that the  $\ell_1$-norm of the inverse z-transform of $M$ is bounded if $M\in\mathbf{RL}_\infty$. We say the multiplier $\bar{M}$ is causal if $M\in\mathbf{RH}_\infty$, $\bar{M}$ is anticausal if $M\in\mathbf{RH}_\infty^{-}$, and $\bar{M}$ is noncausal otherwise. See~\cite{Dahleh:95} for further discussion on causality and stability. Henceforth, we will use $M$ for both the operator and its transfer function.

%


A discrete LTI causal system $G$ has the state space realization of ($A$, $B$, $C$, $D$). That is to say, assuming the input and output of $G$ at sample~$k$ are $u_k$ and $y_k$, respectively, and the inner state is denoted as $x_k$, the following relationship is satisfied 
\begin{equation}\label{eq:5}
G:\begin{cases}
x_{k+1}=Ax_k+Bu_k, & \\
y_k=Cx_k+Du_k, & 
\end{cases}
\end{equation}
in short
\begin{equation}
G \sim \left [
\begin{array}{c|c}
A & B\\ \hline
C & D
\end{array}
\right ].
\end{equation}
Its transfer function is given by $G(z)=C(zI-A)^{-1}B+D$, where $z$ is the z-transform of the forward (or left) shift operator. In fact, this notation is not always adopted in the literature since the definition of the z-transform is not uniform in the use of $z$ or $z^{-1}$. See~\cite{Dahleh:95,Young:88}.

The discrete-time version of the KYP lemma will be used to transfer  frequency domain inequalities into LMIs:
\begin{lem}\label{lem:kyp}
	{\it (Discrete KYP lemma, \cite{Rantzer:96})}
	Given $A$, $B$, $M$, with $\det (e^{j\omega}I-A)\neq0$ for $\omega\in\mathds{R}$ and the pair $(A,B)$ controllable, the following two statements are equivalent:
	\begin{enumerate}
		\item[(i)]\begin{align}
		\begin{bmatrix}
		(e^{j\omega}I-A)^{-1}B  \\
		I
		\end{bmatrix}^*M
		\begin{bmatrix}
		(e^{j\omega}I-A)^{-1}B  \\
		I
		\end{bmatrix} \leq 0.
		\end{align}
		\item[(ii)] There is a matrix $X\in\mathds{R}^{n\times n}$ such that $X=X^\top$ and
		\begin{align}
		M+\begin{bmatrix}
		A^\top XA-X      &   A^\top XB      \\
		B^\top XA    &  B^\top XB	
		\end{bmatrix}\leq0.	
		\end{align}
	\end{enumerate}
	The corresponding equivalence for strict inequalities holds even if the pair $(A,B)$ is not controllable. 
\end{lem}
Throughout this paper, the superscript $^*$ stands for conjugate transpose.

\begin{rem}
	State space representations such as~\eqref{eq:5} are appropriate for causal systems, but not for anticausal and noncausal systems. These can be represented in state space as  descriptor systems. 
	The KYP lemma has been extended to descriptor systems in~\cite{Camlibel} for continuous-time LTI systems. In \cite{Mustapha} an approach to the analysis of discrete singular systems is presented; however it is restricted to causal systems. 
	In this work we exploit the structure of our multipliers to find causal systems that have the same frequency response on the unit circle. Hence the classical KYP lemma suffices.
\end{rem}

The discrete-time Lur'e system is represented in Fig.~\ref{mp1}. The interconnection relationship is
\begin{equation}\label{eq:4}
\begin{cases}
v_k=f_k+(Gw)_k,  & \\
w_k=-\phi (v_k)+g_k. & 
\end{cases}
\end{equation}
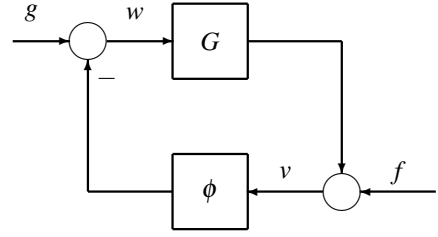
\begin{figure}[t!]
	\centering
	\ifx\JPicScale\undefined\def\JPicScale{1}\fi
	\unitlength \JPicScale mm
	\begin{picture}(55,30)(0,0)
	\linethickness{0.3mm}
	\put(-1.25,25){\line(1,0){7.5}}
	\put(6.25,25){\vector(1,0){0.12}}
	\linethickness{0.3mm}
	\put(8.75,5){\line(1,0){11.25}}
	\linethickness{0.3mm}
	\put(8.75,5){\line(0,1){17.5}}
	\put(8.75,22.5){\vector(0,1){0.12}}
	\linethickness{0.3mm}
	\put(11.25,25){\line(1,0){8.75}}
	\put(20,25){\vector(1,0){0.12}}
	\linethickness{0.3mm}
	\put(30,25){\line(1,0){12.5}}
	\linethickness{0.3mm}
	\put(42.5,7.5){\line(0,1){17.5}}
	\put(42.5,7.5){\vector(0,-1){0.12}}
	\linethickness{0.3mm}
	\put(30,5){\line(1,0){10}}
	\put(30,5){\vector(-1,0){0.12}}
	\linethickness{0.3mm}
	\put(42.5,5){\circle{5}}
	
	\linethickness{0.3mm}
	\put(45,5){\line(1,0){10}}
	\put(45,5){\vector(-1,0){0.12}}
	\put(50,7.5){\makebox(0,0)[cc]{$f$}}
	
	\put(35,7.5){\makebox(0,0)[cc]{$v$}}
	
	\put(15,28.75){\makebox(0,0)[cc]{$w$}}
	
	\linethickness{0.3mm}
	\put(20,30){\line(1,0){10}}
	\put(20,20){\line(0,1){10}}
	\put(30,20){\line(0,1){10}}
	\put(20,20){\line(1,0){10}}
	\linethickness{0.3mm}
	\put(20,10){\line(1,0){10}}
	\put(20,0){\line(0,1){10}}
	\put(30,0){\line(0,1){10}}
	\put(20,0){\line(1,0){10}}
	\put(25,5){\makebox(0,0)[cc]{$\phi$}}
	
	\put(25,25){\makebox(0,0)[cc]{$G$}}
	
	\linethickness{0.3mm}
	\put(8.75,25){\circle{5}}
	
	\put(11.25,20){\makebox(0,0)[cc]{$-$}}
	
	\put(1.25,28.75){\makebox(0,0)[cc]{$g$}}
	
	\end{picture}
	
	\vskip-2mm
	\caption{Lur'e problem}\label{mp1}
	\vskip-6mm
\end{figure}

The system~\eqref{eq:4} is well-posed if the map $(v,w)\mapsto(g,f)$ has a causal inverse on $\ell\times\ell$, and this feedback interconnection is $\ell_{2}$-stable if for any $f,g\in\ell_2$, both $w,v\in\ell_{2}$.

The memoryless nonlinearity $\mathbf{\phi}:\md{R}\mapsto\md{R}$ with $\phi(0)=0$ is said to be bounded if there exists $C$ such that $|\phi(x)|<C|x|$ for all $x\in\md{R}$ and $\phi$ is said to be monotone if for any two real numbers $x_1$ and $x_2$ then
\begin{equation}\label{eq:12}
0\leq\frac{\phi(x_1)-\phi(x_2)}{x_1-x_2}.
\end{equation}
Moreover, $\phi$ is slope-restricted in the interval $S[0,K]$, henceforth $\phi_K$, if
\begin{equation}\label{eq:11}
0\leq\frac{\phi_K(x_1)-\phi_K(x_2)}{x_1-x_2}\leq K,
\end{equation}
for all $x_1\neq x_2$. Finally, the nonlinearity $\phi$ is said to be odd if $\phi(x)=-\phi(-x)$ for all $x\in\md{R}$.


%
%

Zames--Falb multipliers preserve the positivity of the class of monotone nonlinearities~\cite{Zames:68,Willems:68}. Then a loop transformation allows us to obtain the following result for slope restricted nonlinearities: 

\begin{thm}
	[\cite{Willems:71,Willems:68}]\label{Thrm:1} 
	Consider the feedback system in Fig.~\ref{mp1} with $G\in\mathbf{RH}_\infty$, and $\phi$ is a slope-restricted in $S[0,K]$. Suppose that there exists a multiplier $M:\ell_2(\md{Z})\mapsto\ell_2(\md{Z})$ whose impulse response is $m:\md{Z}\mapsto\md{R}$ and satisfies $\sum_{k=-\infty}^\infty |m_k|<2m_0$, 
	\begin{equation}\label{eq:14}
	\hbox{Re}\left\{M(z)(1+KG(z))\right\}>0 \quad \forall |z|=1,
	\end{equation}
	and either $m_k\leq0$ for all $k\neq0$ or $\phi$ is also odd. Then the feedback interconnection~\eqref{eq:4} is $\ell_2$-stable.  
\end{thm}

The above theorem leads to the definition of the class of Zames--Falb multipliers:
\begin{defn}\label{def:zfm}
	{\it (DT LTI Zames--Falb multipliers \cite{Willems:68})} The class of discrete-time SISO LTI Zames--Falb multipliers contains all LTI convolution operators $M:\ell_2(\md{Z})\mapsto\ell_2(\md{Z})$ whose impulse response is $m:\md{Z}\mapsto\md{R}$ satisfies $\sum_{k=-\infty}^\infty |m_k|<2m_0$. Without loss of generality, the value of $m_0$ can be chosen to be 1.  
\end{defn}
\begin{rem}
An important subclass of Zames--Falb multipliers is obtained by adding the limitation $m_k\leq0$, which must be used if we only have information about slope-restriction of the nonlinearity.
\end{rem}
\begin{rem}
	It is also standard to write Definition~\ref{def:zfm} using the $\ell_1$-norm by stating the condition as $\|M\|_1<2$.
\end{rem}

\begin{defn}
	{\it (Nyquist value)} Given $G\in\mathbf{RH}_\infty$, the Nyquist value $k_N$ is the supremum of all the positive real numbers $K$ such that $\tau K G(z)$ satisfies the Nyquist Criterion for all $\tau\in[0,1]$. It can also be expressed as:
	\begin{equation}
	k_N=\sup\{K \in \mathds{R}^+:\inf_\omega\{|1+\tau kG(e^{j\omega})|\}>0),\forall \tau \in[0,1]\}.
	\end{equation}
	In terms of its state space realization~\eqref{eq:5}, $k_N$ is the supremum of $K$ such that all eigenvalues of ($A-BKC$)  are located in the open unit disk, with $K$ in the interval $[0,k_N]$.
\end{defn}

\begin{rem}
	The Kalman conjecture is not valid for discrete-time systems even for plants of order~2 \cite{Carrasco:15,Heath:15}. There is no {\em a priori} guarantee (except for first order systems) that if $k$ is less than the Nyquist value for the plant then the negative feedback interconnection of the plant and a nonlinearity slope-restricted in $S[0,k]$ is stable.
\end{rem}


\section{Searches for IIR Multipliers}


In~\ref{sec:res01} we present a search for discrete-time causal multipliers that is the counterpart to the search for continuous-time causal multipliers presented in \cite{Turner:2009} (see also \cite{Carrasco:2012}). In Section~\ref{sec:res02} we present the anticausal counterpart, similar in spirit to the continuous-time anticausal search of~\cite{Carrasco:2012a}. The results in this section were fully presented in~\cite{Ahmad:13}, so proofs are omitted.



When the multiplier is parameterised in terms of its state-space representation as in \cite{Turner:2009,Carrasco:2012}, we require the following bound \cite{Rieber:2008} for all the searches.
\begin{lem}[\cite{Rieber:2008}] \label{lem:1} Consider a dynamical system $G$ represented by 
	(\ref{eq:5}) and $x_0=0$. Suppose that there exist $\mu>0$, $0<\lambda<1$ and $P=P^\top$ such that
	\begin{equation}
	\begin{bmatrix}
	A^\top P A-\lambda P & A^\top P B \\
	\star                & B^\top P B-\mu I  \\
	\end{bmatrix}<0,
	\end{equation}
	\begin{equation}
	\begin{bmatrix}
	(\lambda-1) P+ C^\top C  & C^\top D \\
	\star       & (\mu-\gamma^2)I+D^\top D \\
	\end{bmatrix}< 0.
	\end{equation}
	Then $\|G\|_1<\gamma$. Furthermore, $A$ has all its eigenvalues in the open unit disk. 
\end{lem}

The use of this result is a fundamental limitation of this method as the parameterisation of the multipliers requires their causality to be established before carrying out the search. Another important feature of this method is that it requires the nonlinearity to be odd as it is not possible to ensure the positivity of the impulse response of the multiplier.

\subsection{Causal search}\label{sec:res01}
In the spirit of~\cite{Turner:2009}, a search over the class of causal discrete-time Zames--Falb multipliers is presented as follows:
\begin{prop}
	\label{p1}
	Let 
	\begin{equation*}
	G(z) \sim
	\left [
	\begin{array}{c|c}
	A_g  &   B_g       \\ \hline
	C_g  &   D_g
	\end{array}
	\right ]
	\end{equation*}
	where $A_g \in \mathbb{R}^{n \times n}$, $B_g \in \mathbb{R}^{n \times 1}$, $C_g \in \mathbb{R}^{1 \times n}$ and $D_g \in \mathbb{R}^{1 \times 1}$. Let $\phi_k$ be an odd nonlinearity slope-restricted in $S[0,K]$. Without loss of generality, assume that the feedback interconnection of $G$ and a linear gain $K$ is stable. Define $A_p$, $B_p$, $C_p$ and $D_p$ as follows:
	\begin{align}
	A_p&=A_g; \label{G_inv1}\\
	B_p&=B_g; \label{G_inv2} \\
	C_p&=kC_g; \label{G_inv3} \\
	D_p&=1+kD_g. \label{G_inv4}
	\end{align}
	
	Assume that there exist positive definite symmetric matrices $S_{11}>0$, $P_{11}>0$, unstructured matrices $\mathbf{\hat{A}}$, $\mathbf{\hat{B}}$ and $\mathbf{\hat{C}}$ with the same dimension as $A_u$, $B_u$, and $C_u$, respectively, and positive constants  $0<\mathbf{\mu}<1$ and $0<\mathbf{\lambda}<1$ such that the LMIs \eqref{M_1}, \eqref{M_2}, and \eqref{M_3} (given on the following page) are satisfied. Then the feedback interconnection~\eqref{eq:0} is $\ell_2$-stable.
	\begin{figure*}[t!]
		\begin{center}
			LMIs in Proposition~\ref{p1}:
		\end{center}
		\begin{align}
		&\begin{bmatrix}
		-S_{11}      &     \star          &       \star       &    \star    &     \star   \\
		-S_{11}      &    -P_{11}      &       \star       &    \star    &     \star   \\
		-C_p-\mathbf{\hat{C}}   &     -C_p         &   -D_p^\top-D_p    &    \star    &     \star   \\
		S_{11}A_p    &   S_{11}A_p     &    S_{11}B_p   &  -S_{11}&   \star   \\
		P_{11}A_p+\mathbf{\hat{B}}C_p+\mathbf{\hat{A}} & P_{11}A_p+\mathbf{\hat{B}}C_p & P_{11}B_p+\mathbf{\hat{B}}D_p &   -S_{11}  &   -P_{11}
		\end{bmatrix}<0, \label{M_1} \\
		&\begin{bmatrix}
		\lambda(S_{11}-P_{11})  &   \star     &   \star   \\
		0             &  -\mathbf{\mu} I   &   \star   \\
		-\mathbf{\hat{A}}            &   -\mathbf{\hat{B}}   &   S_{11}-P_{11}
		\end{bmatrix}<0,  \label{M_2} \\
		&\begin{bmatrix}
		(\mathbf{\lambda}-1)(P_{11}-S_{11})  &          \star           &       \star     \\
		0             &     (\mu-1)I      &       \star     \\
		\mathbf{\hat{C}}     &      0            &        -I
		\end{bmatrix}<0. \label{M_3}
		\end{align}
		\hrule
	\end{figure*}
\end{prop}

\begin{rem}Similar to the continuous case, the inequalities~\eqref{M_1},~\eqref{M_2}, and~\eqref{M_3} are not LMIs if $\lambda$ is defined as variable. Hence, the use of this result requires a linear search of $\lambda$ over the interval between $0$ and $1$.
\end{rem}

\begin{rem} The change of variable is the same as in the continuous case. Therefore the multiplier can be recovered following~\cite{Carrasco:2012} using
	\begin{eqnarray}
	A_u &=& -(P_{11}-S_{11})^{-1}\mathbf{\hat{A}}, \\
	B_u &=& -(P_{11}-S_{11})^{-1}\mathbf{\hat{B}}, \\
	C_u &=& \mathbf{\hat{C}}.
	\end{eqnarray}
\end{rem}

\begin{rem}
	Under further conditions, e.g. $D_p=0$, it is possible to extend this method with a first order anticausal component in the multiplier, i.e.
	$M(z)=(1+m_{-1}z)+M_c(z),$
	under the constraint $|m_{-1}|<1$.
	The development of the result is similar with the use of the state-space representation of $zG(z)$.
\end{rem}	

\subsection{Anticausal multiplier}\label{sec:res02}

The anticausal counterpart of the above search can be stated as follows:

\begin{prop}
	\label{p2}
	Let $G\in\mathbf{RH}_\infty$ be represented in the state space by $A_g$, $B_g$, $C_g$ and $D_g$ where $A_g \in \mathbb{R}^{n \times n}$, $B_g \in \mathbb{R}^{n \times 1}$, $C_g \in \mathbb{R}^{1 \times n}$ and $D_g \in \mathbb{R}^{1 \times 1}$. Let $\phi_K$ an odd nonlinearity slope-restricted in $S[0,K]$. Without loss of generality, assume that the feedback interconnection of $G$ and a linear gain $k$ is well-posed and stable. Define $A_p$, $B_p$, $C_p$ and $D_p$ as follows:
	\begin{align}
	A_p&=A_g-B_g(kD_g+1)^{-1}KC_g; \\
	B_p&=-B_g(KD_g+1)^{-1} ; \\
	C_p&=(KD_g+1)^{-1}kC_g;  \\
	D_p&=(KD_g+1)^{-1}.
	\end{align}
	Assume that there exist positive definite symmetric matrices $S_{11}>0$, $P_{11}>0$, unstructured matrices $\mathbf{\hat{A}}_u$, $\mathbf{\hat{B}}_u$ and $\mathbf{\hat{C}}_u$, and positive constants  $0<\mathbf{\mu}<1$ and $0<\mathbf{\lambda}<1$ such that the LMIs \eqref{M_1}, \eqref{M_2}, and \eqref{M_3} are satisfied, then the feedback interconnection~\eqref{eq:0} is $\ell_2$-stable.
\end{prop}
%

\begin{rem}
	Once the search has provided the matrices $A_u$, $B_u$, and $C_u$, then the multiplier is given by:
	\begin{equation}
	M_{ac}(z)=C_u\left(z^{-1}I-A_u\right)^{-1}B_u+1,
	\end{equation}
	which can be written as
	\begin{equation}
	M_{ac}(z)\sim \left [
	\begin{array}{c|c}
	A_u^{-\top}  &   A_u^{-\top}C_u^\top       \\ \hline
	B_u^\top A_u^{-\top}  &   1-B_u^\top A_u^{-\top}C_u^\top
	\end{array}
	\right ],
	\end{equation}
	if $A_u$ is non-singular. If $A_u$ is singular, then the result is still valid but the multiplier does not have a forward representation. Note that the region of convergence of this transfer function does not include $z=\infty$ and the term $m_0$ in the inverse z-transform of $M_{ac}(z)$ corresponds with $M_{ac}(0)$, i.e. $(\mathcal{Z}^{-1}(M_{ac}))(0)=M_{ac}(0)$.
\end{rem}


\section{Searches for FIR multipliers}

In this section, we restrict our attention to FIR multipliers, i.e. 
\begin{equation}
M(z)=\sum_{i=-n_f}^{n_b}m_iz^{-i},\label{M:definition}
\end{equation}
where $n_b\geq0$ and $n_f\geq0$. 
Without loss of generality we set $m_0=1$. If the nonlinearity is not odd we consider only the subclass of Zames--Falb multipliers with $m_i\leq0$ for all $i\in\mathds{Z}\backslash \{0\}$. 
The multiplier $M$ is said to be causal if $n_b\geq0$ and $n_f=0$, it is said to be anticausal if $n_b=0$ and $n_f\geq0$, and it is said to be noncausal if $n_b>0$ and $n_f>0$. 

Two different searches are included as they provide alternative insights on the design of the multiplier. To conclude the section, we show that any Zames--Falb multiplier can be phase-substituted by an appropriate FIR multiplier.

\subsection{Hard-Factorizations of Zames--Falb multipliers}

In this section we develop an LMI search for FIR Zames--Falb multipliers. In Lemma \ref{lemma:1} we show that the $\ell_1$ condition can be expressed with linear constraints. In Lemma \ref{Theorem:1} we show that although our multiplier is noncausal, the positivity condition can be expressed in terms of a nonsingular state-space representation, leading to an LMI formulation. Our main stability result is stated in Theorem~\ref{main:theorem}. It is possible to show that the LMI requires a positive definite matrix, so it is a hard-factorization.

We seek a Zames--Falb multiplier $M(z)$ with structure of~(\ref{M:definition}) and $m_0=1$
such that
\begin{align}
\mbox{Re}\{ M(z) (1+KG(z))\} > 0\mbox{ for all } |z|=1.\label{cond00}
\end{align}

\begin{lem}\label{lemma:1}
	If $M(z)$ has the structure of (\ref{M:definition}) with $m_0=1$, then
	$M(z)$ is a Zames--Falb multiplier provided
	\begin{equation}
	m_{i}\leq 0 \mbox{ for }i=-n_f,\ldots,-1\mbox{ and }i=1,\ldots,n_b,\label{lemma1:case1a}
	\end{equation}
	and
	\begin{equation}
	\sum_{i=-n_f}^{n_b} m_i > 0.\label{lemma1:case1b}
	\end{equation}
	If the nonlinearity is odd then we can write $m_{i} = m^+_{i} - m^-_{i}$ for $i=-n_f,\ldots, n_b$ (we define $m_0^+=1$ and $m_0^-=0$) and $M(z)$ is a Zames--Falb multiplier provided:
	\begin{equation}
	m^+_{i} \geq 0 \mbox{ and }
	m^-_{i}\geq 0
	\mbox{ for }i=-n_f,\ldots,n_b,\label{lemma1:case2a}
	\end{equation}
	and
	\begin{equation}\
	\sum_{i=-n_f}^{n_b} m^+_i +\sum_{i=-n_f}^{n_b} m^-_i < 2.\label{lemma1:case2b}
	\end{equation}
\end{lem}
\begin{proof}
	This follows immediately from Theorem~\ref{Thrm:1}. The decomposition for odd nonlinearities is the Jordan measure decomposition (e.g. \cite{Billingsley:95}).
\end{proof}
\begin{rem}
	If the nonlinearity is not odd this leads to $n_f+n_b+1$ linear constraints while if the nonlinearity is odd this leads to $2n_f+2n_b+1$ linear constraints. 
\end{rem}

Given $P(z)=1+kG(z),$ condition (\ref{cond00}) can be written:
\begin{align}
M(z) P(z) + \left [ M(z)P(z)\right ]^* > 0\mbox{ for all } |z|=1.
\label{cond00a}
\end{align}

However, since  $M$ is noncausal and $P\in\mathbf{RH}_\infty$, it follows that $MP$ does not have a nonsingular state-space description. This is addressed in Lemma~\ref{Theorem:1} below. 

First we define some quantities.
Let $P(z)$ have state-space description
\begin{equation}
P \sim \left [
\begin{array}{c|c}
A_p & B_p\\ \hline
C_p & D_p
\end{array}
\right ],\label{P:state}, 
\end{equation}
where $A\in\mathds{R}^{n_p\times n_p}$.
Let $n=\max(n_f,n_b)$ and define 
\begin{align}
\tilde{A} = 
\left [
\begin{array}{ccc}
A_p & B_p & 0\\
0 & 0 & I_{n-1}\\
0 & 0 & 0
\end{array}
\right ]
\mbox{ and }
\tilde{B} = 
\left [
\begin{array}{c}
0\\
0 \\
1
\end{array}
\right ].
\label{def:tildeA}
\end{align}
where $\tilde{A}\in\mathds{R}^{(n_p+n)\times (n_p+n)}$. Also let
\begin{align}
C_n & = \left [
\begin{array}{ccc}
C_p & D_p & 0_{n-1}
\end{array}
\right ],
\label{def:Cn}
\end{align}
and
\begin{align}
C_{d,i} &=   \left [
\begin{array}{ccc}
0_{n_p+n-i} & 1 &  0_{i-1}
\end{array}
\right ]\mbox{ for }i = 1,\ldots,n_f,\label{def:Cdi}
\end{align}
where $n_p$ is the dimension of $A_p$.
Define $C_i$ as
\begin{align}
C_i & = 	  C_n \tilde{A}^{n-i} + \sum_{j=1}^{-i}\left (
C_n\tilde{A}^{n-i-j-1}\tilde{B}
\right )
C_{d,j}
\mbox{ for } i=-n_f,\ldots -1,
\label{def:Cneg}\\
C_0 & = 	  C_n \tilde{A}^{n},\label{def:Czero}\\
C_i & = 
C_n \tilde{A}^{n-i}\mbox{ for } i=1,\ldots,n_b,\label{def:Cpos}
\end{align}
and $D_i$ as
\begin{align}
D_{i} & = C_n \tilde{A}^{n-i-1}\tilde{B}\mbox{ for } i=-n_f,\ldots,-1,\label{def:Dneg}\\
D_0 & = C_n \tilde{A}^{n-1}\tilde{B},\label{def:Dzero}\\
D_i & = 0 \mbox{ for } i=1,\ldots,n_b.\label{def:Dpos}
\end{align}

Then we can say:
\begin{lem}\label{Theorem:1}
	Suppose $P(z)$ is a causal and stable discrete-time transfer function with state-space description (\ref{P:state})
	and suppose $M(z)$ is a noncausal FIR transfer function given by (\ref{M:definition}) with $m_0=1$.
	There exist $P_i(z)$ for $i=-n_f,\dots,n_b$ with nonsingular state-space representation such that
	\begin{align}
	M(z) P(z) + [M(z)P(z)]^* = &  \sum_{i=-n_f}^{n_b} m_{i}\left (P_{i}(z) + [P_{i}(z)]^*\right )\nonumber\\
	& \mbox{ for all }|z|=1.\label{condtion:20}
	\end{align}
	
	Furthermore the statement
	\begin{equation}
	M(z)P(z)+ \left [M(z)P(z)\right ]^* >0 \mbox{ for all } |z|=1,
	\end{equation}
	is equivalent to the statement that there exists a matrix $X\in\mathds{R}^{(n_p+n)\times (n_p+n)}$ such that $X=X^\top$ and
	\begin{align}
	\left [
	\begin{array}{cc}
	\tilde{A}^\top X\tilde{A}-X & \tilde{A}^\top X\tilde{B}\\
	\tilde{B}^\top X\tilde{A} & \tilde{B}^\top X\tilde{B}
	\end{array}
	\right ]
	-M_f^\top \Pi M_f <0,\label{condition:23}
	\end{align}
	with
	\begin{align}
	\Pi & = \left [
	\begin{array}{cc}
	0 & m\\
	m^\top & 0
	\end{array}
	\right ],\\
	m^\top & = \left [
	\begin{array}{ccccccc}
	m_{-n_f}, & \dots, & m_{-1}, & 1, & m_1, & \ldots  & m_{n_b}
	\end{array}
	\right ], \\
	\mbox{and}\nonumber\\
	M_f & =
	\left [
	\begin{array}{cc}
	C_{-n_f} & D_{-n_f} \\
	\vdots & \vdots\\
	C_{n_b} & D_{n_b}\\
	0 & 1
	\end{array}
	\right ],
	\end{align}
	with $\tilde{A}$, $\tilde{B}$, $C_i$ and $D_i$  given by (\ref{def:tildeA}), (\ref{def:Cneg}-\ref{def:Cpos}) and (\ref{def:Dneg}-\ref{def:Dpos}).
\end{lem}

\begin{proof}
	
	We can write
	\begin{align}
	M(z)P(z) = \sum_{i=-n_f}^{n_b} m_i z^{-i} P(z).
	\end{align}
	Hence we must choose causal $P_i(z)$ for $i=-n_f,\ldots,n_b$ such that 
	\begin{align}
	P_i(z) + \left [P_i(z)\right ]^* = z^{-i}P(z) + \left [z^{-i}P(z)\right ]^*\mbox{ for all }|z|=1.
	\end{align}
	It follows immediately that for $i=0,\ldots,n_b$ we can choose
	\begin{align}
	P_i(z) = z^{-i}P(z).
	\end{align}
	When $i$ is negative, $z^{-i}P(z)$ is not causal (beware: if $i$ is negative then $z^{-i}$ is anticausal). We can partition $z^{-i}P(z)$ into causal and anticausal parts
	\begin{align}
	z^{-i} P(z) = P^{AC}_i(z)+P^C_i(z).
	\end{align}
	The partition is standard since $P^{AC}_i$ is FIR (e.g. \cite{Dahleh:95}). If  we write $P$ as
	\begin{align}
	P(z) = \sum_{k=0}^{\infty}p_kz^{-k},
	\end{align}
	then, for $i=-n_f,\ldots,-1$, we have
	\begin{align}
	P^{AC}_i(z) & = \sum_{k=0}^{-i-1}p_kz^{-i-k}\nonumber\\
	& = D_pz^{-i} + \sum_{k=1}^{-i-1}C_pA_p^{k-1}B_pz^{-i-k},
	\end{align}
	and
	\begin{align}
	P^C_i(z)  & = z^{-i}P(z)-P^{AC}_i(z)\nonumber \\
	& = C_p A_p^{-i}(zI-A_p)^{-1}B_p + C_pA_p^{-i-1}B_p.
	\end{align}
	Then we can choose
	\begin{align}
	P_i(z) = P^C_i(z) + P^{AC}_i(z^{-1}).
	\end{align}
	We parameterize each $P_i(z)$ as follows. Let $n=\max(n_f,n_b)$. Define $\tilde{A}$ and $\tilde{B}$ as (\ref{def:tildeA}) and $C_n$ as (\ref{def:Cn}). Then 
	\begin{align}
	z^{-n}P(z) = C_n (zI-\tilde{A})^{-1}\tilde{B}.
	\end{align}
	When $i$ is positive we can write
	\begin{align}
	P_i(z) & = z^{-i} P(z)\nonumber\\
	& =  C_n \tilde{A}^{n-i}(zI-\tilde{A})^{-1}\tilde{B}\nonumber\\
	& = C_i(zI-\tilde{A})^{-1}\tilde{B}+D_i\mbox{ for } i=1,\ldots,n_b,
	\end{align}
	where $C_i$ and $D_i$ are given by (\ref{def:Cpos}) and (\ref{def:Dpos}) respectively.
	Similarly
	\begin{align}
	P_0(z) & = P(z)\nonumber\\
	& = C_n \tilde{A}^{n}(zI-\tilde{A})^{-1}\tilde{B} + C_n \tilde{A}^{n-1}\tilde{B}\nonumber\\
	& = C_0(zI-\tilde{A})^{-1}\tilde{B}+D_0,
	\end{align}
	where $C_0$ and $D_0$ are given by (\ref{def:Czero}) and (\ref{def:Dzero}) respectively.
	
	When $i$ is negative, we write
	\begin{align}
	P_i(z) = &  C_p A_p^{-i}(zI-A_p)^{-1}B_p + C_pA_p^{-i-1}B_p \nonumber\\
	& + D_p z^{-i} + \sum_{k=1}^{-i-1}C_pA_p^{k-1}B_pz^{-i-k}.
	\end{align}
	The state space realization of the delay operator $z^{-j}$ is formulated as 
	\begin{align}
	z^{-j} = C_{d,j} (zI-\tilde{A})^{-1}\tilde{B},
	\label{def_delay}
	\end{align}
	with $C_{d,i}$ given by (\ref{def:Cdi}). So we can write this
	\begin{align}
	P_i(z)  = &  C_n \tilde{A}^{n-i}(zI-\tilde{A})^{-1}\tilde{B} + C_n\tilde{A}^{n-i-1}\tilde{B} \nonumber\\
	& + C_n \tilde{A}^{n-1}\tilde{B} z^{-i} + \sum_{k=1}^{-i-1}C_n \tilde{A}^{n+k-1}\tilde{B}z^{-i-k}\nonumber\\
	= & C_i(zI-\tilde{A})^{-1}\tilde{B}+D_i\mbox{ for } i=-n_f,\ldots,-1,
	\end{align}
	where $C_i$ and $D_i$ are given by (\ref{def:Cneg}) and (\ref{def:Dneg}) respectively.
	
	Finally  we can write
	\begin{align}
	M(z)& P(z)+  [M(z)P(z)]^* \nonumber\\
	& =
	\left [
	\begin{array}{c}
	P_{-n_f}(z)\\
	\vdots \\
	P_{n_b}(z)\\
	1
	\end{array}
	\right ]^*
	\left [
	\begin{array}{cc}
	0 & m\\
	m^\top & 0
	\end{array}
	\right ]
	\left [
	\begin{array}{c}
	P_{-n_f}(z)\\
	\vdots \\
	P_{n_b}(z)\\
	1
	\end{array}
	\right ]\nonumber\\
	& =
	\left [
	\begin{array}{c}
	(zI-\tilde{A})^{-1}\tilde{B}\\
	1
	\end{array}
	\right ]^*
	M_f^\top
	\Pi
	M_f
	\left [
	\begin{array}{c}
	(zI-\tilde{A})^{-1}\tilde{B}\\
	1
	\end{array}
	\right ].
	\end{align}
	The result then follows immediately from the KYP Lemma for discrete-time systems (Lemma~\ref{lem:kyp}).
\end{proof}


We can now state our main result.
\begin{thm}\label{main:theorem}
	Consider the feedback system in Fig.1 with $G\in\mathbf{RH}_\infty$, and $\phi$ is a nonlinearity slope-restricted in $S[0,k]$. 
	Suppose we can find $m$ and $X$ such that the LMI (\ref{condition:23}) is satisfied under the conditions of Lemma~\ref{Theorem:1} with the additional constraints either (\ref{lemma1:case1a}) and (\ref{lemma1:case1b})  or $\phi$ is also odd and (\ref{lemma1:case2a}) and (\ref{lemma1:case2b}).  Then the feedback interconnection~\eqref{eq:4} is $\ell_2$-stable. 
\end{thm}
\begin{proof}
	This follows immediately from Lemma~\ref{lemma:1}, Lemma~\ref{Theorem:1} and Theorem~\ref{Thrm:1}.
\end{proof}

\begin{rem}
	Theorem~\ref{main:theorem} gives an LMI condition for stability. The symmetric matrix $X$ has $(n+n_p)(n+n_p+1)/2$ independent parameters while the parameter vector $m$ has $n_f+n_b$ free variables when the nonlinearity is not odd and $2n_b+2n_f$ free variables when the nonlinearity is odd. When the nonlinearity is not odd there are $n_f+n_b+1$ linear constraints on $m$ and when the nonlinearity is odd there are $2n_f+2n_b+1$ linear constraints.
\end{rem}

\begin{prop} If there exists $X=X^T$ satisfies \eqref{condition:23} in Lemma~\ref{Theorem:1}, then $X>0$.  
\end{prop}
\begin{proof}
	It follows since the diagonal matrix block $M_f^T\Pi M_f$ with the $(n+n_p)$ first rows and columns is zero, hence  condition~\eqref{condition:23} requires $\tilde{A}^TX\tilde{A}-X<0$ with all eigenvalues of $\tilde{A}$ in the open unit disk, hence $X>0$.
\end{proof}

\subsection{Alternative implementation of FIR search}

In this section we provide a causal-factorization approach which is widely discrete-time for general robust techniques~\cite{Hosoe:13}, but here we focus on Zames--Falb multipliers. One can think of this technique as the discrete-time counterpart of factorization approach in~\cite{Veenman:16} for general continuous-time multipliers.

By the IQC theorem, we seek a Zames--Falb multiplier such that
\begin{equation*}\label{eq:8}
\begin{bmatrix} -G(z) \\ 1\end{bmatrix}^* \begin{bmatrix} 0 & K M^*(z)\\ K M(z) & -(M(z)+M^*(z))\end{bmatrix}\begin{bmatrix} -G(z) \\ I\end{bmatrix}>0 \quad \forall |z|=1.
\end{equation*}
Substituting the Zames--Falb multiplier $M(z)$ by its FIR form~\eqref{M:definition} with $n_b=n_f=n$, then the IQC multiplier can be factorized via lifting as follows
\begin{equation*}\label{eq:6}
\begin{bmatrix} 0 & K M^*(z)\\ K M(z) & -(M(z)+M^*(z))\end{bmatrix}=\Psi(z)^*\kappa(k,m)\Psi(z),
\end{equation*}
where 
$$\Psi(z)=\begin{bmatrix} 1 & 0\\ z^{-1} & 0 \\ z^{-2} & 0\\ \vdots & \vdots \\ z^{-n} & 0\\ 0 &1 \\ 0&z^{-1}  \\ 0& z^{-2} \\ \vdots & \vdots \\ 0&  z^{-n}\\ \end{bmatrix},$$
and $\kappa(k,m)$ is given in~\eqref{eq:K} in next page.
\begin{figure*}[h!]
	\begin{equation}\label{eq:K}
	\kappa (k,m)=\begin{bmatrix} 0 & 0 & 0 &\cdots & 0 & k m_0 & k m_{1} & k m_{2} & \cdots & k m_{n}\\
	0 & 0 & 0 &\cdots & 0 & k m_{-1} & 0 & 0 & \cdots &0\\
	0 & 0 & 0 &\cdots & 0 & km_{-2} & 0 & 0 & \cdots &0\\
	\vdots & \vdots & \vdots&\cdots & \vdots & \vdots & \vdots &  \vdots& \cdots &\vdots\\
	0 & 0 & 0 &\cdots & 0 & km_{-n} & 0 & 0 & \cdots &0\\
	km_0 & km_{-1} & km_{-2} &\cdots & km_{-n} & -2m_0 & -m_1-m_{-1} & -m_2-m_{-2} & \cdots & -m_n-m_{-n}\\
	km_{1} & 0 & 0 &\cdots & 0 & -m_1-m_{-1} & 0 & 0 & \cdots &0\\
	km_{2} & 0 & 0 &\cdots & 0 & -m_2-m_{-2} & 0 & 0 & \cdots &0\\
	\vdots & \vdots & \vdots&\cdots & \vdots & \vdots & \vdots &  \vdots& \cdots &\vdots\\
	km_{n} & 0 & 0 &\cdots & 0 & -m_n-m_{-n} & 0 & 0 & \cdots &0\\
	\end{bmatrix}
	\end{equation}
	\hrule	
\end{figure*}

\begin{thm}\label{main:II}
	Consider the feedback system in Fig.1 with $P\in\mathbf{RH}_\infty$, and $\phi$ is a nonlinearity slope-restricted in $S[0,K]$. Let 
	$$\Psi(z)\begin{bmatrix}
	-G(z)\\1
	\end{bmatrix}\sim\left [
	\begin{array}{c|c}
	\hat{A}  &  \hat{B}       \\ \hline
	\hat{C}  &   \hat{D}
	\end{array}
	\right ],$$
	and $$ \quad m^\top  = 
	\begin{bmatrix}m_{-n}, & \dots, & m_{-1}, & 1, & m_1, & \ldots  & m_{n}\end{bmatrix}.	
	$$ If there exist $X=X^T$ and $m$ such that
	\begin{equation}\label{eq:7}
	\begin{bmatrix}
	\hat{A}^\top X\hat{A}-X & \tilde{A}^\top X\tilde{B}\\
	\tilde{B}^\top X\tilde{A} & \tilde{B}^\top X\tilde{B}
	\end{bmatrix}
	+ \begin{bmatrix}
	\hat{C} & \hat{D}
	\end{bmatrix}^T \kappa(k,m)\begin{bmatrix}
	\hat{C} & \hat{D}
	\end{bmatrix} <0,
	\end{equation} 
	\begin{equation}
	\sum_{i=-n}^{n} |m_i| < 2,
	\end{equation}
	and either $m_i\leq0$ for all $i\neq0$ or $\phi$ is odd, then the feedback interconnection~\eqref{eq:4} is $\ell_2$-stable. 
\end{thm}
\begin{proof}
	The proof follows by the application of the KYP lemma, as~\eqref{eq:7} is equivalent to \eqref{eq:14}; hence the conditions of Theorem~\ref{Thrm:1} hold, and stability is then guaranteed. 
\end{proof}
\begin{rem} Conditions for quasi-odd, quasi-monotone nonlinearities~\cite{Heath:15cdc} can be straightforwardly implemented. 
\end{rem}

\begin{rem} In this factorization, it is not possible to ensure $X>0$. The introduction of the condition $X>0$ would reduce the class of available multipliers.
\end{rem}

\subsection{Phase-Equivalence}

In the spirit of~\cite{Carrasco:13,Carrasco:14}, we can state the phase-equivalence between the full class of LTI Zames--Falb multipliers and FIR Zames--Falb multipliers as follows:

\begin{lem}\label{lem:pe}
	Given $P\in\mathbf{RH}_\infty$, if there exists a Zames--Falb multiplier $M$ such that
	\begin{equation}\label{eq:14a}
	\hbox{Re}\left\{M(z)P(z)\right\}>0 \quad \forall |z|=1,
	\end{equation}
	then there exists an FIR Zames--Falb multiplier $M_{\text{FIR}}$ such that
	\begin{equation}\label{eq:14b}
	\hbox{Re}\left\{M_{\text{FIR}}(z)P(z)\right\}>0 \quad \forall |z|=1.
	\end{equation}
\end{lem}

\begin{proof} Given an LTI Zames--Falb multiplier
	\begin{equation}
	M(z)=\sum_{i=-\infty}^{\infty}m_iz^{-i}, \quad \text{and}\quad\sum_{i=-\infty}^{\infty}|m_i|<2m_0,
	\end{equation}
	for any $\varepsilon>0$, there exists $N$ such that  
	\begin{equation}\label{N_ineq}
	\sum_{i=-\infty}^{-N-1}|m_i|+\sum_{i=N+1}^{\infty}|m_i|<\varepsilon.
	\end{equation}
	We can write
	\begin{equation}
	M(z)=\sum_{i=-N}^{N}m_iz^{-i}+M_t(z)=M_{\text{FIR}}(z)+M_t(z),
	\end{equation}
	with $\|M_t\|_\infty\leq\|M_t\|_1<\varepsilon$.
	
	Meanwhile, as $P(z)$ and $M(z)$ are continuous functions in the unit circle, by the extreme value theorem~\cite{Rudin:64}, there exists $\delta_1>0$ such that
	\begin{equation}\label{eq:14c}
	\hbox{Re}\left\{M(z)P(z)\right\}\geq\delta_1 \mbox{ for all } |z|=1.
	\end{equation}
	Let us choose $N$ such that (\ref{N_ineq}) is satisfied with $\varepsilon=\frac{\delta_1}{2\|P\|_\infty}$. Then for all $z$ satisfying $|z|=1$ we find
	\begin{align}
	\hbox{Re}\left\{M(z)P(z)\right\} &  =\hbox{Re}\left\{M_{\text{FIR}}(z)P(z)\right\}+\hbox{Re}\left\{M_t(z)P(z)\right\}\nonumber\\
	& \leq
	\hbox{Re}\left\{M_{\text{FIR}}(z)P(z)\right\}+|M_t(z)P(z)|\nonumber\\
	& \leq \hbox{Re}\left\{M_{\text{FIR}}(z)P(z)\right\}+|M_t(z)||P(z)|\nonumber\\
	& \leq
	\hbox{Re}\left\{M_{\text{FIR}}(z)P(z)\right\}+\|M_t\|_\infty\|P\|_\infty\nonumber\\
	& \leq \hbox{Re}\left\{M_{\text{FIR}}(z)P(z)\right\}+\frac{\delta_1}{2},\label{eq:14e}
	\end{align}
	Finally, rearranging using~\eqref{eq:14e} and using~\eqref{eq:14a}, it follows that
	\begin{align}
	\hbox{Re}\left\{M_{\text{FIR}}(z)P(z)\right\} & \geq \hbox{Re}\left\{M(z)P(z)\right\}-\frac{\delta_1}{2}\nonumber\\
	& \geq\frac{\delta_1}{2}>0 \mbox{ for all } |z|=1.\label{eq:14d}
	\end{align}
\end{proof}


\section{Numerical results}

\subsection{Comparison with other results}


Table I presents the numerical examples that we analyse. All six plants are taken from previous papers~\cite{Ahmad:15,Heath:15}. Results are shown in Table II. We have run results in Theorem~\ref{main:theorem} for values of $n=n_b=n_f$ between 1 and 100, and optimal results are presented in Table II indicating $n^*$ the optimal value of n.

The FIR search is significantly better than all competitive results in the literature, it beats classical searched as the Tsypkin Criterion~\cite{Tsypkin:1962,Kapila:96} as well as the most recent result in the Lyapunov literature~\cite{Ahmad:15,Park:2018}. It is worth highlighting that these Lyapunov methods correspond with particular cases of FIR Zames--Falb multipliers, besides small numerical discrepancies. Results~\cite{Ahmad:15} corresponds with the case $n_b=n_f=1$, whereas results in~\cite{Park:2018} correspond with the case $n_b=n_f=2$, besides small numerical discrepancies. Results have been obtained by using CVX~\cite{gb08,cvx} with the SeDuMi solver~\cite{Sturm:99}. 

Roughly speaking, the higher the order of the multiplier, the better the results. However, there is a small deterioration due to numerical issues as $n$ increase. We show that the maximum slope suffers also a small deterioration as $n$ increases by including the values of the maximum slope with $n_b=nf=100$. Figure~\ref{figure:6} shows this deterioration as $n$ increases for Example 1.  We associate this deterioration to the numerical error associated with an increment in the size of the matrices in the LMIs.  

\begin{figure}[htbp]
	\centering
	\includegraphics[width=\linewidth]{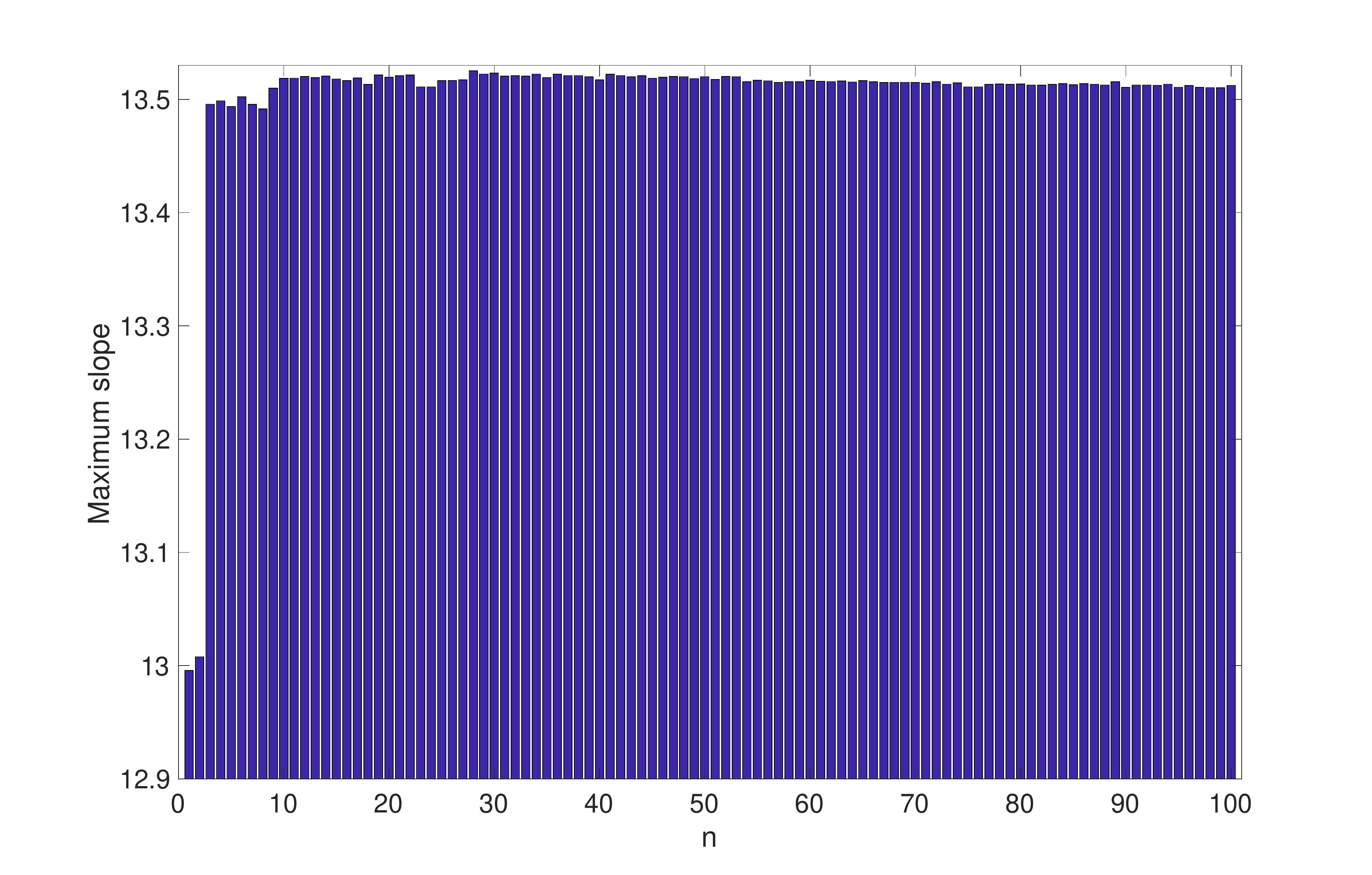}
	\caption{Maximum slope for Example 1 as $n=n_f=n_b$ increases.}\label{figure:6}
\end{figure}

There are small numerical differences between results with both factorizations. In general, there is a slightly better performance of the hard factorization presented in Section~IV.A.  For instance, maximum slope in Example 1 is 13.5215 with $n_f=n_b=28$, whereas the soft factorization in Section IV.B reaches 13.5162 with $n=11$. Similar deterioration is observed, maximum slope is reduced to 13.5001 when $n=100$.

As expected, results for odd nonlinearities are always better than results for non-odd nonlinearities. Although it is natural as the set of multiplier increase and phase retrictitions are reduced, this contrasts with the SISO results reported in \cite{chen95a} for the continuous case. In Examples~1 to~4 the FIR results beat all others in the literature. In Example~5 both the FIR results and others in the literature achieve the Nyquist value. Example 6 is used in~\cite{Heath:15} to show that stability is deteriorated by the lack of symmetry. From~\cite{Heath:15}, we expect that a maximum slope above 1 for odd nonlinearities and below 1 for non-odd nonlinearities.  



\begin{table}[t]
	\centering  \caption{Examples}\label{t:1}\vskip-2mm
	\small{\begin{tabular}{|l|l|}
			\hline
			Ex.                & Plant   \\ \hline \hline
			1~\cite{Ahmad:15} & $G_1(z)=\frac{0.1z}{z^2-1.8z+0.81}$ \\
			2~\cite{Ahmad:15}           & $G_2(z)=\frac{z^3 - 1.95 z^2 + 0.9 z + 0.05}{z^4 - 2.8 z^3 + 3.5 z^2 - 2.412 z + 0.7209}$ \\
			3~\cite{Ahmad:15}           & $G_3(z)=-\frac{z^3 - 1.95 z^2 + 0.9 z + 0.05}{z^4 - 2.8 z^3 + 3.5 z^2 - 2.412 z + 0.7209}$ \\
			4~\cite{Ahmad:15}          & $G_4(z)=\frac{ z^4 - 1.5 z^3 + 0.5 z^2 - 0.5 z + 0.5}{4.4 z^5 - 8.957 z^4 + 9.893 z^3 - 5.671 z^2 + 2.207 z - 0.5}$ \\
			5~\cite{Ahmad:15}          & $G_5(z) =\frac{-0.5 z + 0.1}{z^3 - 0.9 z^2 + 0.79 z + 0.089}$\\
			6~\cite{Heath:15}          & $G_6(z) =\frac{2z+0.92}{z^2-0.5z}$\\
			\hline
		\end{tabular}
		\vskip-2mm}
\end{table}

\begin{table*}[t]
	\centering\caption{Slope-restricted results by using different stability criteria.}\vskip-2mm
	\begin{tabular}{ |l|l|c|c|c|c|c|c|}
		\hline
		\textbf{Criterion }                              & Odd nonlinearity? &  Ex. 1   & Ex. 2  &Ex. 3      &Ex. 4       & Ex. 5 &Ex. 6 \\ \hline
		Circle Criterion \cite{Tsypkin:1962}  & N          & 0.7934 & 0.1984 &  0.1379  &  1.5312    &    1.0273  & 0.6510  \\  \hline
		Tsypkin Criterion \cite{Kapila:96}     & N        & 3.8000  & 0.2427 &  0.1379  &   1.6911    &   1.0273  & 0.6510  \\  \hline
		Ahmad et. al. (2015), Thm 1 \cite{Ahmad:15}   & N    & 12.4178 & 0.72614 &  0.30267  &    2.5911  & 2.4475 & 0.9067   \\  \hline
		Park et al. (2018)& N & 12.9960 & 0.7396 & 0.3054 & 2.5904 &  2.4475& 0.9108 \\ \hline \hline
		Causal DT Zames--Falb     & Y   & 12.4355  & 0.7687  & 0.2341   & 3.3606      &  2.3328 & 0.9222 \\  \hline	
		Anticausal DT Zames--Falb  & Y  & 1.4994   & 0.4816  & 0.3058    &   3.2365     & 2.4474& 1.0869  \\  \hline
		FIR Zames--Falb ($n_f=1$, $n_b=1$)          & N                   	& 12.9957      & 0.7397     &    0.3054     &     2.5904      &   2.4475 & 0.9108  \\  \hline
		FIR Zames--Falb ($n_f=1$, $n_b=1$)          & Y                   	&  12.9957    &   0.7783    &   0.3076     &     3.1350       &   2.4475& 1.0869  \\ \hline
		FIR Zames--Falb ($n_f=2$, $n_b=2$)          & N                   	& 12.9957      & 0.7397     &    0.3054     &     2.5904      &   2.4475 & 0.9115 \\  \hline
		FIR Zames--Falb ($n_f=2$, $n_b=2$)          & Y                   	&  12.9957    &   0.7783    &   0.3076     &     3.1350       &   2.4475&  1.0869  \\ \hline
		FIR Zames--Falb ($n_f=100$, $n_b=100$)          & N                   	&   13.0280   &  0.7948   &   0.3113   &    3.8234      &   2.4475 &  0.9115 \\  \hline
		FIR Zames--Falb ($n_f=100$, $n_b=100$)          & Y                   	&   13.5124   &    1.1047   &  0.3115    &     3.8196     & 2.4469  &  1.0849  \\ \hline
		FIR Zames--Falb  ($n_f=n_b=n^*$)         & N                   	& 13.0284 ($6$)   & 0.8015 ($12$)  & 0.3120 ($12$)   &  3.8240 ($24$)     &  2.4475 ($1$) & 0.9115 ($2$) \\  \hline
		FIR Zames--Falb        ($n_f=n_b=n^*$)   & Y                   	& 13.5251 ($28$)  & 1.1073 ($7$)  & 0.3126 ($4$)    &   3.8304 ($7$)    &  2.4475 ($1$)& 1.0869 ($1$) \\  \hline \hline
		Nyquist Value                            			  & N/A    & 36.1000   & 2.7455    & 0.3126     & 7.9070       &  2.4475   & 1.0870 \\  \hline	
	\end{tabular}
	\label{t:2}\vskip-1mm
\end{table*}

\subsection{Structure of Multipliers}

It is worth highlighting the sparsity in the structure of the multiplier. In Figure~\ref{figure:5}, we show the terms above $10^{-5}$. The structure of the multiplier can be explained as it reaches it maximum allowed phase over some particular range of frequencies when it has an sparse structure~\cite{Wang:18}, therefore the optimization use only the positions in the multiplier which are useful to correct the phase of the $(1+kG)$ in the region when it is not positive.

\begin{figure}[htbp]
	\centering
	\includegraphics[width=\linewidth]{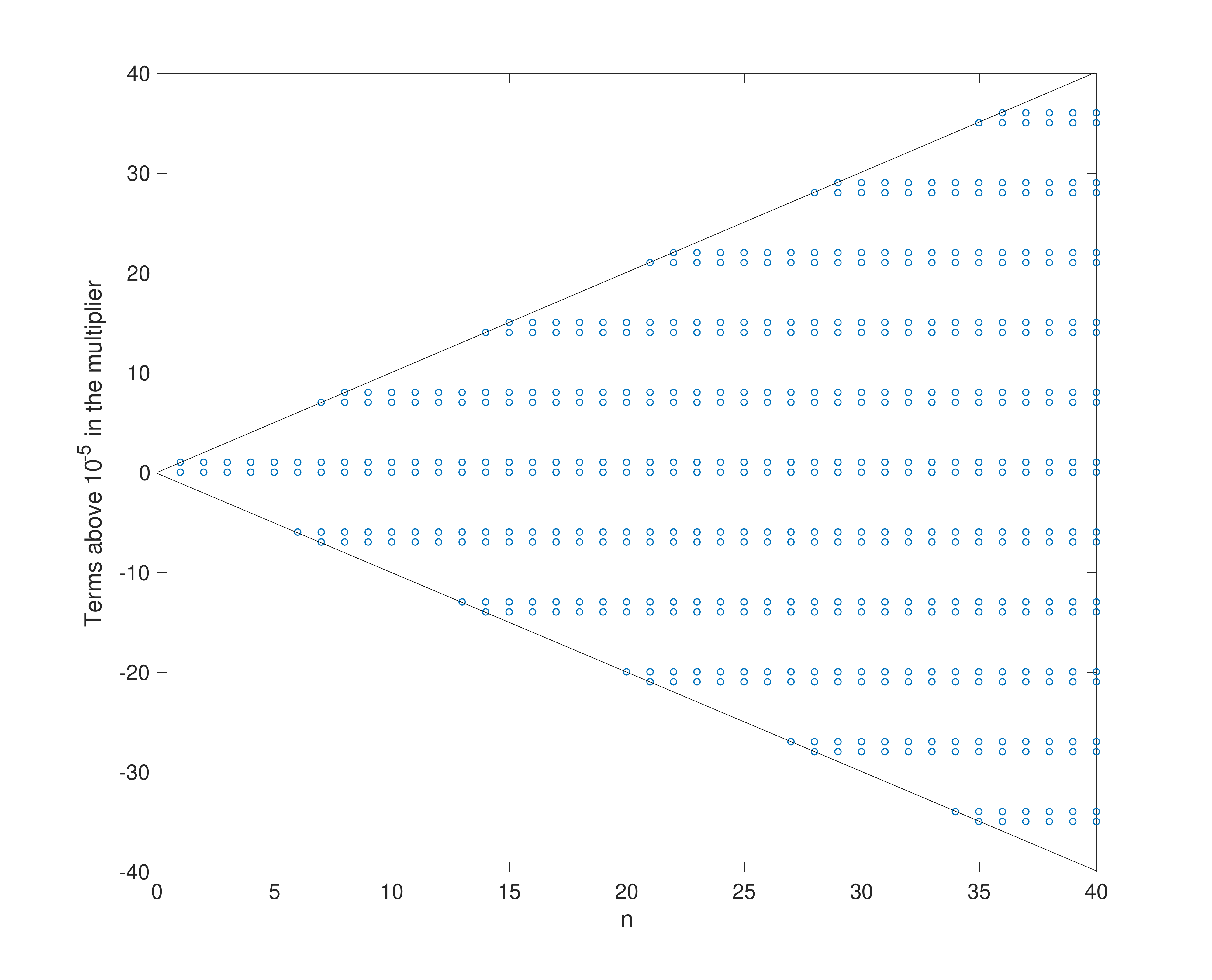}
	\caption{Pattern of the significant terms of the multipliers for Example 1 as $n=n_f=n_b$ increases.}\label{figure:5}
\end{figure}

\subsection{Computational time}


It is interesting to analyse the performance of the search as $n$ increases. As expected, the computational time increases in a polynomial fashion. However, it is worth highlighting that the use of the Jordan measure decomposition in \eqref{lemma1:case2b} increases the computational time as the number of variables in the multiplier is doubled. The code is run in MATLAB R2017a with Mac Book Pro 2.3 GHz Intel Core i5 and 8GB 2133 MHz LPDDR3.

\begin{figure}[htbp]
	\centering
	\includegraphics[width=\linewidth]{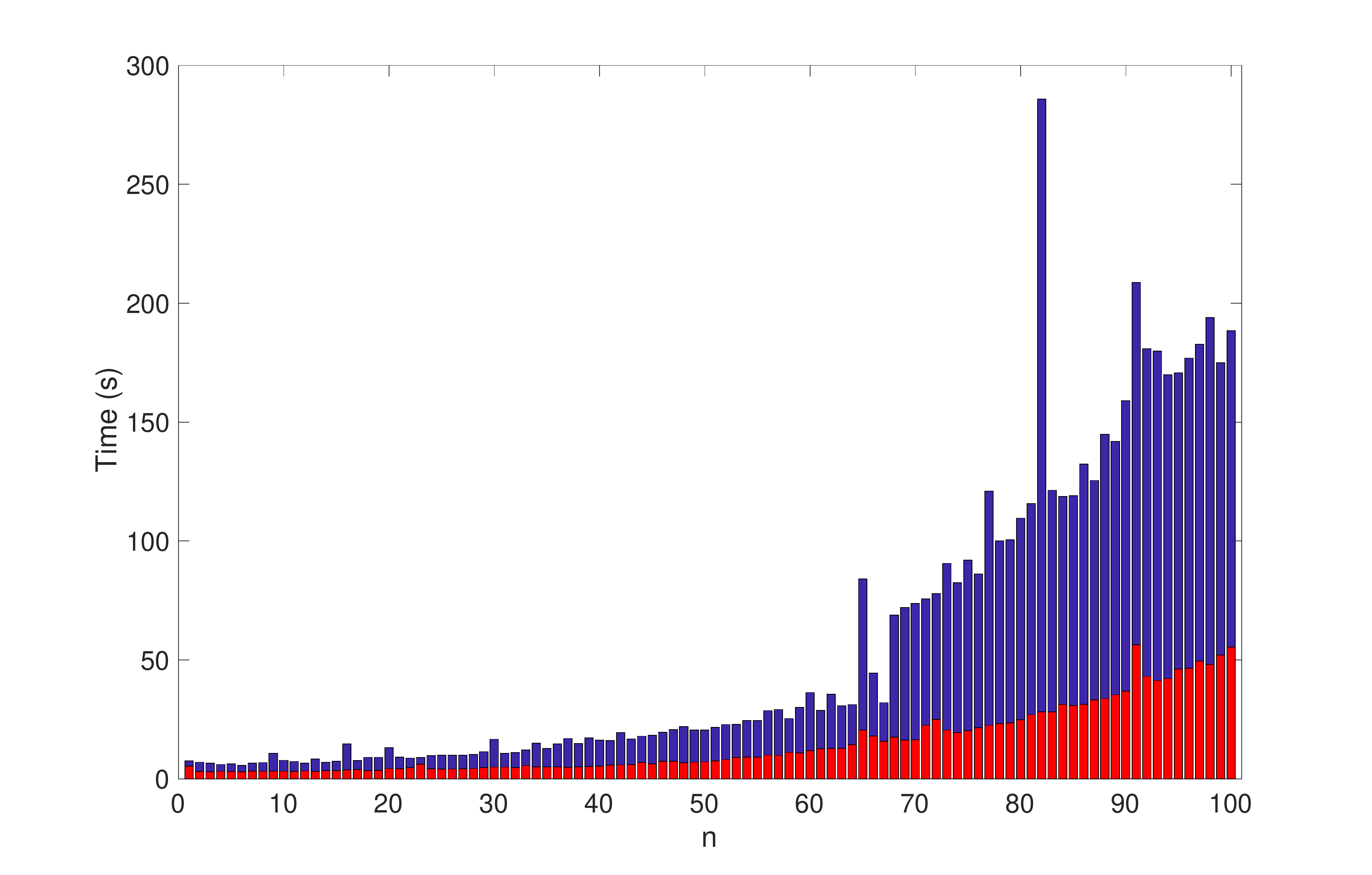}
	\caption{Computational time require to find the maximum slope in Example 1 with a precision of $10^{-5}$ in the bisection algorithm. The bisection method is started with $k_{\min}=0$ and $k_{max}=1.1 k_{N}$. The case $m_i\leq0$ in red (slope-restricted nonlinearities), and the in blue the most general class of multipliers (slope-restricted and odd nonlinearities). }
\end{figure}

\section{Application to Safonov's method}

Safonov proposed the first numerical method to search for Zames--Falb multipliers~\cite{Safonov87}. Different modifications have been proposed~\cite{Gapski:94,Chang:2012} to produce numerical optimization of the multiplier. In this section, we provide a different approach, which require manual tuning from the user, but may be used to test the conservatism of fully-autonomous numerical searches. Note that other manual tunings of rational multipliers have been suggested in the literature~\cite{Veenman:16,Tugal:17}, which also lead to improvements over fully-autonomous searches.

The idea is straightforward. Given a continuous plant $G(s)$ we find the maximum slope as follows:
\begin{enumerate}
	\item Choose a sampling time $T_s$ and find the discrete-time counterpart $G_d(z)$.
	\item Choose $n_f$ and $n_b$. Find the discrete-time Zames--Falb multiplier
	\[M_d(z) = \sum_{i=-n_f}^{n_b} m_i z^{-i},\]
	corresponding to the maximum $K_d$ such that
	\[\mbox{Re}\left \{M_d(z)(1+K_dG_d(z))\right \}>0 \mbox{ for all }|z|=1.\]
	\item (Optional) Choose $\epsilon>0$. For $-n_f\leq i\leq n_b$, if $|m_i|<\epsilon$, set $m_i=0$ for tractability.
	\item Define
	\[M(s) = \sum_{i=-n_f}^{n_b} m_i e^{-iT_s s}.\]
	It follows immediately that $M(s)$ belongs to the appropriate class of Zames--Falb multipliers. 
	\item Find the maximum $K$ such that
	\[\mbox{Re}\left \{M(s)(1+KG(s))\right \}>0 \mbox{ for all }\mbox{Re}\left \{s\right \}=0.\]
\end{enumerate}

\subsection*{Numerical results}
All the following results are taken from~\cite{Carrasco:14} and given in Table~\ref{tab:1}. Here we just provide details of the suitable multiplier obtained by the above method. We have used standard command in MATLAB \emph{c2d} to perform the discretization.
We use $\epsilon=10^{-3}$ in Step 3. A summary of the results is given in Table~\ref{tab:6}, but we provide detailed information for each example.
%

\begin{table}
	\begin{tabular}{|l|l|}
		\hline
		Ex. & $G(s)$\\
		\hline
		\hline
		
		1 & $G_1(s)=\frac{s^2-0.2s-0.1}{s^3+2s^2+s+1}$\\
		
		2 & $G_2(s)=-G_1(s)$\\
		
		3 & $G_3(s)=\frac{s^2}{s^4+0.2s^3+6s^2+0.1s+1}$\\
		
		4 & $G_4(s) = -G_3(s)$\\
		
		
		5 & $G_5(s) = \frac{s^2}{s^4+0.0003s^3+10s^2+0.0021s+9}$\\
		
		6 & $G_6(s) = -G_5(s)$\\
		7 & $G_7(s) = \frac{s^2}{s^3+2s^2+2s+1}$\\
		
		
		8 & $G_8(s) =  \frac{9.432(s^2+15.6s+147.8)(s^2+2.356s+56.21)(s^2-0.332s+26.15)}{(s^2+2.588s+90.9)(s^2+11.79s+113.7)(s^2+14.84s+84.05)(s+8.83)}$\\
		
		
		9 & $G_9(s) = \frac{s^2}{s^4+5.001s^3+7.005s^2+5.006s+6}$\\
		
		\hline
	\end{tabular}
	\vskip2mm
	\caption{Continuous-time examples from~\cite{Carrasco:14}}\label{tab:1}
\end{table}

\begin{table*}
	\centering
	\begin{tabular}{|c|c|c|c|c|c|c|c|c|c|}
		\hline
		& Ex.1 & Ex. 2 & Ex. 3 & Ex. 4 & Ex. 5 & Ex. 6 & Ex. 7 & Ex. 8 & Ex. 9\\
		\hline \hline
		Best results in~\cite{Carrasco:14}  & 4.5949 & 1.0894 & 1.6122& 1.2652 & 0.00333 & 0.00333 & {\bf 10,000+} & 87.3854 & 91.0858\\
		\hline
		Algorithm in Section VI    & 4.5949 & 1.0894 & 1.945 & 1.29 & 0.0055 & 0.0039 & Unreliable & Unreliable & 91.0858\\
		\hline \hline
		Nyquist value & 4.5894 & 1.0894 & $\infty$ & 3.5000 & $\infty$ & 1.7142 & $\infty$ & 87.3854 & $\infty$\\
		\hline
	\end{tabular}
\vskip 2mm
\caption{Comparison between best results reported in~\cite{Carrasco:2012a} and continuous time method in Section VI.}\label{tab:6}
\end{table*}

\paragraph{Example 1}
Choose $T_s=0.05$, $N_f=1$, $N_b=1$. The discrete search leads then to the continuous-time multiplier given by 
\[M(s) = -0.5436 e^{0.05s} +1 -0.4561e^{-0.05s}.\]
The multiplier reaches the Nyquist value in this example (K=4.5984) which matches the best results reported in~\cite{Carrasco:2012a}.

\paragraph{Example 2}
Choose $T_s=0.05$, $N_f=0$, $N_b=1$. The discrete search leads then to the continuous-time multiplier given by  
\[M(s) = 1 -0.9551e^{-0.05s}.\]
The multiplier reaches the Nyquist value in this example (K=1.0894) which matches the best results reported in~\cite{Carrasco:2012a}. 

\paragraph{Example 3}
Choose $T_s=0.1$, $N_f=20$, $N_b=0$. The discrete search leads then to the continuous-time multiplier given by 
\[M(s) = 1 -0.6507e^{1.9s}-0.3493e^{2s}.\]
The multiplier reaches $K=1.945$, a 21\% improvement over the best results reported in~\cite{Carrasco:2012a}.

\paragraph{Example 4}
Choose $T_s=0.02$, $N_f=1$, $N_b=80$. The discrete search leads then to the continuous-time multiplier given by  
\[M(s) = -0.9186 e^{0.02s}+1-0.0809e^{-1.6s}.\]
The multiplier reaches $K=1.29$, a 2\% improvement over the best results reported in~\cite{Carrasco:2012a}.

\paragraph{Example 5}
Choose $T_s=0.02$, $N_f=0$, $N_b = 50$. The discrete search leads then to the continuous-time multiplier given by   
\[M(s) = 1 - 0.8902e^{-0.02s}+0.1087e^{-s}.\]
The multiplier reaches $K=0.0055$, a 65\% improvement over the best results reported in~\cite{Carrasco:2012a}.

\paragraph{Example 6}
Choose $T_s=0.02$, $N_f=50$, $N_b = 0$. The discrete search leads then to the continuous-time multiplier given by   
\[M(s) = 1 - 0.7909e^{0.02s}+0.2090e^{s}.\]
The multiplier reaches $K=0.0039$, a 20\% improvement over the best results reported in~\cite{Carrasco:2012a}.

\paragraph{Example 7}
For this example the method is poor. We must sample at $T_s<0.0002$ to achieve a Nyquist value of over 10,000. But with $T_s$ so small, we require $N_f$ and $N_b$ intractably large to obtain good multipliers. For example, choosing $T_s=0.0001$, and $N_f=N_b=50$ gives a maximum $k=  28.6$. By contrast, setting $T_s=0.001$ gives a maximum $k= 768$.  Setting $T_s=0.01$ sets it back to $k=147$.

\paragraph{Example 8}
Again for this example the method is poor. Extreme care must be taken when discretizing the model. Setting $T_s=0.001$ and $N_b=N_f=40$ yields a maximum $k=64$. Other methods yield the Nyquist value, which is circa 87.

\paragraph{Example 9}
Choose $T_s=0.01$, $N_f=70$, $N_b = 1$. The discrete search leads then to the continuous-time multiplier given by  
\begin{equation}\label{mult9}
M(s)=1-0.976e^{-0.01s}-0.0013e^{0.48s}-0.0227e^{0.7s}.
\end{equation} 
The multiplier reaches $K=360$, a 395\% improvement over the best results reported in~\cite{Carrasco:2012a}. Figure~\ref{fig:6} shows that the phase of $M(s)(1+360G_9(s))$ is in the interval $(-90,90)$.

\begin{figure}[htbp]
	\centering
	\includegraphics[width=\linewidth]{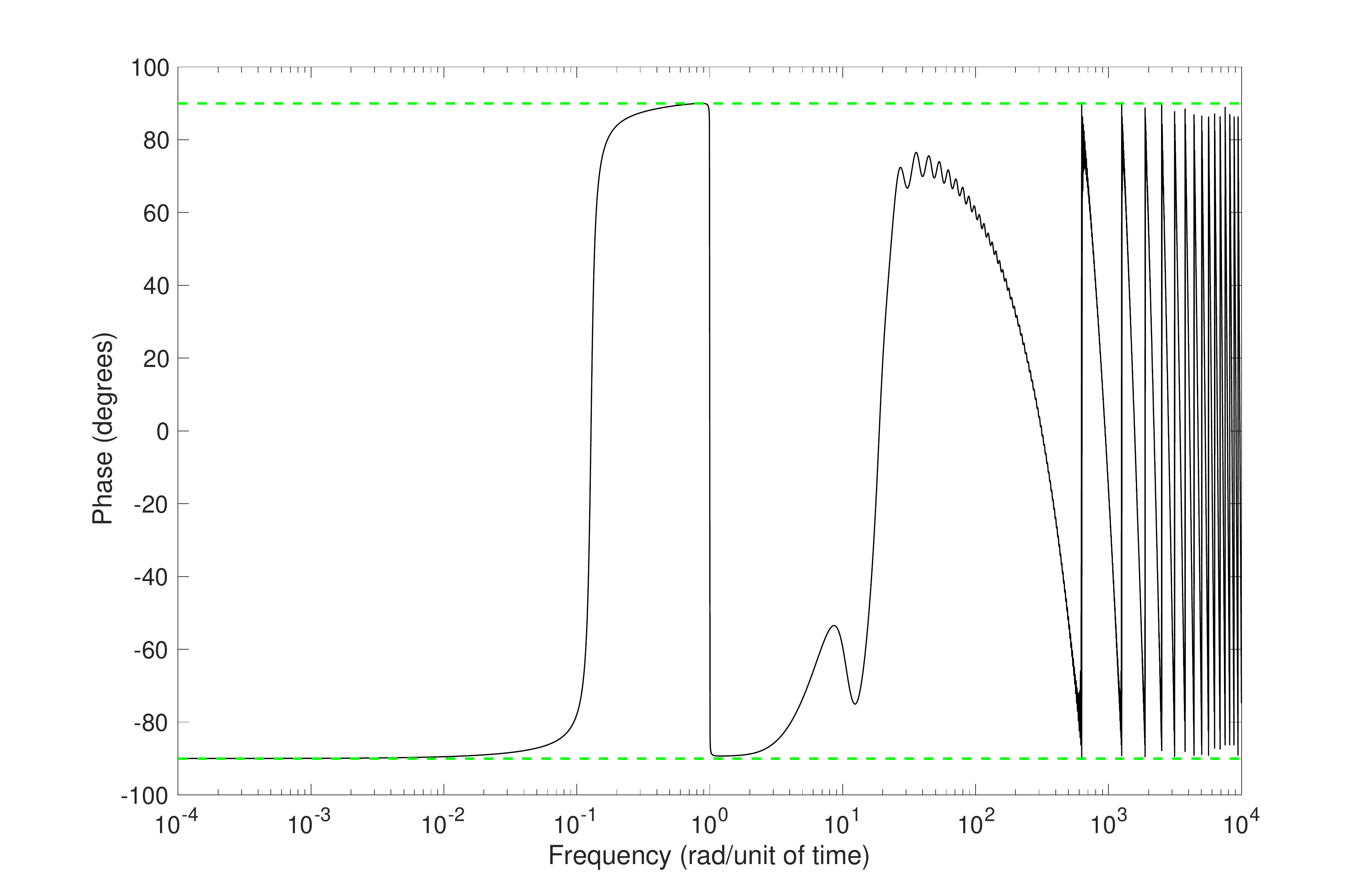}
	\caption{Phase of $M(s)(1+360 G_9(s))$ where $M(s)$ is given by~\eqref{mult9}.}\label{fig:6}
\end{figure}

%
%


\section{Conclusions}

The results in this paper provide the best results in the literature for absolute stability of discrete-time LTI systems in feedback interconnection with slope-restricted nonlinearities. We have developed two search methodologies for discrete-time Zames--Falb multiplier: IIR and FIR. In contrast with continuous-time domain, one of the available searches is better for all examples. We show the superiority of these searches with respect to the recent method based on Lyapunov functions, whose results are similar to our search with $n_b=n_f=2$. Finally, we have extended the results to be used as a tunable search of continuous time Zames--Falb multipliers. The results shows the conservativeness of current state-of-the-art searches over the class of Zames--Falb multipliers.

\end{document}